\documentclass[12pt,scrartcl,a4paper]{article}

\usepackage{a4wide}
\usepackage{setspace}

\usepackage{apacite}

\usepackage[affil-it]{authblk}

\RequirePackage[colorlinks,citecolor=blue,urlcolor=blue]{hyperref}

\RequirePackage[colorlinks,citecolor=blue,urlcolor=blue]{hyperref}
\usepackage{amsfonts,amsthm,amsmath,amssymb,mathtools,bbm,mathrsfs,epsfig,natbib}
\usepackage{thmtools}
\usepackage{multirow}
\usepackage{booktabs}
\usepackage{float}
\usepackage{tabularx}
\usepackage{enumitem}
\usepackage{ amssymb }
\usepackage{subcaption}
\usepackage{capt-of}
\usepackage{diagbox}
\usepackage{color,tikz}
\usepackage{cancel}
\usepackage{makecell}
\numberwithin{equation}{section}

\newcommand{\Cov}{{\rm Cov}}

\newcommand{\R}{\mathbb{R}}

\newcommand{\N}{\mathbb{N}}

\newcommand{\Q}{\mathbb{Q}}
\newcommand{\E}{\mathrm {E}}

\newcommand{\LL}{{\cal L}}

\renewcommand{\d}{\mathrm{d}}
\newcommand{\1}{\mathbbm{1}}

\newcommand{\scalar}[2]{\mathchoice{\left\langle #1, #2\right\rangle}{\langle #1, #2\rangle}{\langle #1, #2\rangle}{\langle #1, #2\rangle}}

\newcommand{\oP}{{\scriptstyle\mathcal{O}}_P}

\declaretheorem[refname = {Theorem}]{theorem}

\declaretheorem[]{lemma}

\declaretheorem[]{example}

\declaretheorem[refname = {Assumption,Assumptions}]{assumption}

\def\BB{ \mathfrak{B} }

\renewcommand{\epsilon}{\varepsilon}
\renewcommand{\rho}{\varrho}
\renewcommand{\theta}{\vartheta}

\newcommand\blfootnote[1]{%
  \begingroup
  \renewcommand\thefootnote{}\footnote{#1}%
  \addtocounter{footnote}{-1}%
  \endgroup
}

\makeatletter
\newcommand{\vast}{\bBigg@{4}}
\newcommand{\Vast}{\bBigg@{5}}
\makeatother

\begin{document}
\title{A BHEP test for multivariate normality on incomplete data}
\author[1,2,3]{Daniel Gaigall$^*$}
\author[3]{Philipp Wübbolding}

\affil[1]{Institute for Data-driven Technologies, FH Aachen - University of Applied Sciences, Heinrich-Mußmann-Straße 1, 52428 Jülich, Germany}

\affil[2]{House of Insurance, Leibniz University Hannover, Welfengarten 1, 30167 Hannover, Germany}

\affil[3]{Faculty 09 - Medical Engineering and Technomathematics, FH Aachen - University of Applied Sciences, Heinrich-Mußmann-Straße 1, 52428 Jülich, Germany}

\date{}

\maketitle

\blfootnote{$^*$Corresponding author. Email address: gaigall@fh-aachen.de, daniel.gaigall@insurance.uni-hannover.de}

\begin{abstract}
A  BHEP test  for  the null hypotesis of multivariate normality on the basis of  incomplete data is introduced. Estimators for the underlying unknown parameters in this situation are suggested. The test  uses characteristic functions and circumvents the problem of singular covariance matrix estimates. As the sample size tends to infinity, an almost sure limit of the test statistic is obtained under the null hypothesis and under alternatives. The convergence in distribution under the null hypothesis  is also proved. Critical values can be obtained using a bootstrap procedure. Simulation studies investigate size and power of the test  and confirm the adequacy of the approach. A real data example demonstrates the application of the test.



\end{abstract}

\section{Introduction}

When addressing a question from multivariate statistics, the statistician can be confronted with a sample of incomplete data vectors. Noticing that multivariate normality is a standard assumption, e.g.,  in statistical modelling,  the question whether the underlying distribution belongs to  the family of multivariate normal distributions is of particular interest. Common approaches to deal with the incomplete data case adapt standard statistical procedures developed for complete data. For testing  multivariate normality, we have the usual BHEP test, see \citet{BaringhausHenze}. The usual BHEP test  uses  characteristic functions and a related $L^2$-distance. The test statistic is based solely on  Mahalanobis distances, which makes the procedure invariant under affine transformations of the data vectors and in particular distribution-free. A review of similar testing procedures for the null hypothesis of  multivariate normality can be found in \citet{RePEc:spr:testjl:v:29:y:2020:i:4:d:10.1007_s11749-020-00740-0}. The most naive way to handle  incomplete data is to discard any  data vectors with missing entries (complete-case analysis). This results in a loss of even more information, especially when the sample size is not very large, the dimension of the data  vectors is not very small or the missing data rate is not very low. Moreover,  imputation techniques, such as mean or median imputation, are often used. Anyway, the resulting distribution of the imputed data is different from the true distribution and the application of related statistical procedures requires great care.  In general, there is no guarantee that the procedure works. The paper \citet{Aleksić09122024} revisits the usual BHEP test,  compares those procedures including complete-case analysis and different imputation techniques and suggest a bootstrap procedure to address this issue. Moreover, the work  \citet{Tsatsi01112024} compares  different tests for multivariate normality under various imputation methods by simulation. To demostrate possible problems that arise with the adaption of standard statistical procedures developed for complete data, let us consider the empirical covariance matrix. Remembering that the regularity of this matrix requires that the sample size is larger than the dimension, discarding any incomplete data vectors could lead to a singular estimate and finally to the  inapplicability of the statistical procedure intended. On the other hand,  imputation typically leads to a systematic  underestimation of the true variances and covariances. A more sophisticated alternative to the adaption of a standard statistical procedure developed for complete data is the application of a method tailored to the incompled data case if available. An example of such a method  is given in  \citet{Gaigall-MargHom-Incompl-data}, where a test for the fully nonparametric testing problem of marginal homogeneity on the basis of possibly incomplete paired data is introduced. Unfortunately, the availability of methods designed for incomplete data is very limited. This applies in particular to the testing  problem of multivariate normality.   \citet{TAN2005164} , \citet{yamada2015kurtosis} and \citet{KURITA2022104824}  treat testing for multivariate normality in special incomplete data cases by focusing  on specific shape parameters of the underlying distribution,  namely kurtosis or skewness. We develop a  BHEP test for  multivariate normality especially for the incomplete data case. Similar as the usual BHEP test, our approach based on  characteristic functions and a related $L^2$-distance that can potentially detect any deviation from the null hypothesis. For more power and  flexibility in applications, the test incorporates a projection approach similar as in   \citet{GaigallLiangWu-HilbertIndependence}. To circumvent the problem of singular covariance matrix estimates, our procedure drops the solely use of  Mahalanobis distances. In contrast to the adaption of standard statistical procedures developed for complete data, our approach is always applicable, including in cases where the sample size is small, the dimension of the data  vectors is large or the missing data rate is high. Due to the incomplete data case, the test statistic is not distribution-free, but a bootstrap procedure is available to obtain critical values in practice. Our BHEP test aligns with a series of recent developments of BHEP  testing procedures for the null hypothesis of gaussianity on the basis of advanced data, that are  \citet{HenzeGamero},  \citet{CDHH} and \citet{GaigallWuebbolding_gBM-Test}, where  similar tests but for (complete) Hilbert space valued or functional data are considered. The paper is structured as follows. In Section \ref{tp}, we formulate the testing problem. The incomplete data setting is also introduced there. Parameter estimation in this situation is discussed in  Section \ref{pe}. Specific suggestions for  estimators that satisfy desired properties are also included. In Section \ref{es}, we introduce the test statistic. Alternative representations are available and useful for computational purposes. Asymptotic results as the sample size tends to infinity are proved in Section \ref{as}. More detailed, an almost sure limit of the test statistic is established, available under the null hypothesis and under alternatives, and the limiting null distribution of the test statistic is  stated. The results motivate to use a bootstrap procedure for the approximation of quantiles of the unknown null distribution of the test statistic, see Section \ref{bp}. The method makes critical values available in applications and ensures a practicable implementation of the test. For the investigation of size and power the test, the outcomes of simulation studies are shown and discussed in Section \ref{s}. Empirical rejection rates obtained confirm the adequacy of the procedure under the null hypothesis and under alternatives. Finally, Section \ref{rde} presents a real data example that demonstrates the application of the test. Here, we consider a  dataset on  air quality in New York City, containing measurements of ozone and solar radiation. Note that all proofs are given in Appendix \ref{p}.

\section{Testing problem and  incomplete data setting}\label{tp}

On a probability space, let $X$ be a  random vector of dimension  $d\in\N$, $d\ge 2$,  with values in $\R^d$ and unknown underlying distribution $\LL(X)$. That is the random vector or  distribution of interest. In our notation, vectors are column vectors, related components are written as $x=(x(1),\dots,x(d))^\top$, $x\in\R^d$,  the standard scalar product  is $\scalar{x}{y}=x^\top y=\sum_{j=1}^dx(j)y(j)$, $x,y\in\R^d$, and the usual norm is denoted by $|x|=\sqrt{\scalar{x}{x}}=\sqrt{xx^\top}=\sqrt{\sum_{j=1}^dx(j)^2}$, $x\in\R^d$. Analogous notations are used for matrices. We consider the testing problem of multivariate normality
\begin{equation*}
    H_0: \LL(X)\in\mathfrak L_X \text{ vs. } H_1:\LL(X)\notin\mathfrak L_X,
\end{equation*}
where  
$$\mathfrak L_X=\{\mathcal N_d(\mu,\Sigma);\mu\in\R^d,\Sigma\in\R^{d\times d}~\text{symmetric positive definite}\}$$
is the family of  multivatiate normal distributions of dimension $d$. We suppose the following regularity assumption on the underlying distribution $\LL(X)$.

\begin{assumption}\label{ass1a}
We suppose that $\E(|X|^2)<\infty$ such that  the expectation vector $\mu=\E(X)\in\R^d$  as well as  the covariance matrix $\Sigma=\Cov(X)\in\R^{d\times d}$ of $X$ exist with finite entries. Moreover, we assume that $\Sigma$ is positive definite.
\end{assumption}

Let $I$ be another  random vector of dimension $d$ on   the probability space with values in $\{0,1\}^d$ and underlying distribution $\LL(I)$ that is unknown in general. This random vector or  distribution is not of interest, but it determines the missingness mechanism in the sense that for all $j=1,\dots,d$ it is $I(j)=1$ if and only if $X(j)$ is observed. Given that  $\mathfrak L_I\neq \emptyset$ is a  known  family of  distributions  of  random vectors of dimension $d$  with values in $\{0,1\}^d$, we suppose that
$$\LL(I)\in\mathfrak L_I.$$
The set $\mathfrak L_I$ consists of all missingness mechanisms that are considered or allowed in our incomplete data setting. The probability mass function of $I$ is denoted by 
$$p(a)=P(I=a),~a\in\{0,1\}^d.$$
Let us consider some examples of missingness mechanism that can be allowed  in our incomplete data setting.

\begin{example}\label{exa0}
Let us consider some examples of possible choices of the family of distributions $\mathfrak L_I$.

\medskip 
(a) The most simple case is that $\mathfrak L_I$  has  only one element. This situation corresponds to the case that the distribution of $I$ is known and given.

\medskip
(b) The most obvious next case is that $\mathfrak L_I$ consists of all possible distributions  of  random vectors of dimension $d$  with values in $\{0,1\}^d$. This situation corresponds to the case that we have no restriction or no prior knowledge on the distribution of $I$.

\medskip
(c) A case of potential interest is that $\mathfrak L_I$ consists of all possible distributions  of  random vectors  of dimension $d$  with values in $\{0,1\}^d$ and components that are  independent and identically distributed. Precisely, we have that $I=(I(1),\dots,I(d))^\top$, where $I(1),\dots,I(d)$ are independent, each with the same binomial (Bernoulli) distribution.

\medskip
(d) Another case of potential interest is that $\mathfrak L_I$ consists of all possible distributions  of  random vectors  of dimension $d$  with values in $\{0,1\}^d$ that describe a  monotone missing (drop-out) data meachanism. Precisely, we have that $I=(I(1),\dots,I(d))^\top$ with the property that
$$I(j)=0\Longrightarrow I(j)=\cdots=I(d)=0,~j=1,\dots,d.$$

\end{example}

We suppose the following assumption.

\begin{assumption}\label{ass1b}
We suppose that    the random vectors $X$ and $I$ are independent and that $p({\bf 1}_d)>0$, where ${\bf 1}_d=(1,\dots,1)\in\R^d$, such that  $X$ is potentially completely observable.
\end{assumption}

To formalize the incomplete data setting, we use the  Hadamard-product  $\odot$ as  an operation defined on any space of real matrices, e.g.,  
$$I\odot X=(I(1)X(1),\dots,I(d)X(d))^\top$$
is the  Hadamard-product of $I$ and $X$, that is a random vector with values in $\R^d$. We will use $^{\odot-1}$ as a notation for  the Hadamard-inverse or the entrywise inverse on  the respective space of real matrices. Let us suppose that we observe
$$(I_i,I_i\odot X_i),~i=1,\dots,n,$$
that is a sample of size $n\in\N$ of independent random variables on  the probability space, each with the same distribution as $(I,I\odot X)$. The second components can be regarded as our  sample of incomplete observations, while the first components encode which values are missing. Note that the special  case  $p({\bf 1}_d)=1$  leads to the usual complete data situation and is covered by our setting. Under Assumption \ref{ass1b}, it holds that
\begin{align*}
  \forall x\in\R^d:  P(X\le x)=\frac{P( I={\bf 1}_d,I\odot X\leq x)}{p({\bf 1}_d)},
\end{align*}
where $\le$ is meant component-wise here, which shows that   the distribution $\LL(I,I\odot X)$ uniquely determines the distribution  $\LL(X)$. This shows that it is sufficient to work with the observations $(I_i,I_i\odot X_i)$, $i=1,\dots,n$, for the treatment of the testing problem under Assumption \ref{ass1b}.

\section{Parameter estimation}\label{pe}

The underlying distribution of our sample $\LL(I,I\odot X)$ depends on the unknown expectation vector $\mu$ and covariance matrix $\Sigma$ of $X$ and on the unknown  probability mass function $p$ of $I$, in particular if the null hypothesis $H_0$ is valid. For that reason, estimators  of these unknown parameters are required. Let $\widehat{\mu}_n$ be an estimator of $\mu$ with values in  $\R^d$, let $\widehat{\Sigma}_n$ be an estimator of $\Sigma$ with values in the space of symmetric postive semidefinite matrices in $\R^{d\times d}$ and let $\widehat{p}_n$  be an estimator of $p$ with values in the space of probability mass functions on $\{0,1\}^d$. Here, we suppose that the estimators are given by appropriate measurable functions applied to the sample   $(I_i,I_i\odot X_i)$, $i=1,\dots,n$. Note that we allow singular covariance matrix estimates. The following assumption states desireable properties of the estimators.

\begin{assumption}\label{ass2b}
We suppose the consistency of the estimators 
$$\widehat{\mu}_n\overset{a.s.}{\longrightarrow}  \mu~\text{as}~n\to\infty$$ 
and
$$\widehat{\Sigma}_n\overset{a.s.}{\longrightarrow}   \Sigma~\text{as}~n\to\infty$$
as well as
$$\widehat{p}_n(a)\overset{a.s.}{\longrightarrow}   p(a)~\text{as}~n\to\infty$$
for all $a\in\{0,1\}^d$.
\end{assumption}

Additionally, the following  properties of  the estimators are required.

\begin{assumption}\label{ass3}
We suppose the existence of a measurable map $l_\mu:\R^d\times\R^d\rightarrow \R^d$ with $\E(|l_\mu(I,I\odot X)|^2)<\infty$,  $\E(l_\mu(I,I\odot X))=0$ and
\begin{equation*}
    \sqrt{n}(\widehat{\mu}_n-\mu)=\frac{1}{\sqrt{n}}\sum_{i=1}^n l_\mu(I_i,I_i\odot X_i)+\oP(1)~\text{as}~n\rightarrow\infty,
\end{equation*}
the existence of a measurable map $l_\Sigma:\R^d\times\R^d\rightarrow \R^{d\times d}$ with $\E(|l_\Sigma(I,I\odot X)|^2)<\infty$,  $\E(l_\Sigma(I,I\odot X))=0$ and  
\begin{equation*}
    \sqrt{n}(\widehat{\Sigma}_n-\Sigma)=\frac{1}{\sqrt{n}}\sum_{i=1}^n l_\Sigma(I_i,I_i\odot X_i)+\oP(1)~\text{as}~n\rightarrow\infty
\end{equation*}
and the existence of a measurable map $l_{p(a)}:\R^d\times\R^d\rightarrow [0,1]$ with $\E(|l_{p(a)}(I,I\odot X)|^2)<\infty$,  $\E(l_{p(a)}(I,I\odot X))=0$ and  
\begin{equation*}
    \sqrt{n}(\widehat{p}_n(a)-p(a))=\frac{1}{\sqrt{n}}\sum_{i=1}^n l_{p(a)}(I_i,I_i\odot X_i)+\oP(1)~\text{as}~n\rightarrow\infty
\end{equation*}
for all $a\in\{0,1\}^d$.
\end{assumption}

In what follows, we give examples of  estimators.

\begin{example}\label{exa1}
Let us introduce the following estimators.

\medskip
(a)  An estimator of the expectation vector $\mu$ of $X$  is given by 
\begin{equation*}
    \widehat{\mu}_n=\bigg(\sum_{i=1}^n I_i\odot X_i\bigg)\odot\bigg(\sum_{i=1}^nI_i\bigg)^{\odot-1}
\end{equation*}
with the conventions $1/0=\infty$ and $0\cdot\infty=0$. 

\medskip
(b) Introducing with 
\begin{equation*}
    \widehat{\Lambda}_n=\frac{1}{n}\sum_{i=1}^n(I_i\odot (X_i-\widehat{\mu}_n))(I_i\odot (X_i-\widehat{\mu}_n))^\top
\end{equation*}
a symmetric  matrix, with
\begin{equation*}
    \widehat{M}_n=\frac{1}{n}\sum_{i=1}^nI_iI_i^\top
\end{equation*}
a symmetric  matrix and with
\begin{equation*}
    \widehat{S}_n=\widehat{\Lambda}_n\odot\widehat{M}_n^{\odot-1},
\end{equation*}
a symmetric matrix, let
\begin{equation*}
    \widehat{\Sigma}_n=\sqrt{\widehat{S}_n\widehat{S}_n}
\end{equation*}
be the existing and uniquely determined postive semidefinite matrix root of the symmetric positiv semidefinite matrix $\widehat{S}_n\widehat{S}_n$. Then,  $  \widehat{\Sigma}_n$ is an estimator of  the covariance matrix $\Sigma$ of $X$.

\medskip
(c) Let us consider estimators of the probability mass function $p$ of $I$.  In the situation of Example \ref{exa0} (a), an estimator is given by 
\begin{equation*}
    \widehat{p}_n(a)=p(a),~a\in\{0,1\}^d.
\end{equation*}
In the situation of Example \ref{exa0} (b), an estimator is given by 
\begin{equation*}
    \widehat{p}_n(a)=\frac{1}{n}\sum_{i=1}^n\1(I_i=a),~a\in\{0,1\}^d,
\end{equation*}
where $\1$ denotes the indicator function. In the situation of  Example \ref{exa0} (c), an estimator is given by 
\begin{equation*}
    \widehat{p}_n(a)= \bigg(\frac{1}{n}\sum_{i=1}^n\frac{1}{d}\sum_{j=1}^d\1(I_i(j)=1)\bigg)^{\sum_{j=1}^da(j)}\bigg(1- \frac{1}{n}\sum_{i=1}^n\frac{1}{d}\sum_{j=1}^d\1(I_i(j)=1)\bigg)^{d-\sum_{j=1}^da(j)}
\end{equation*}
for $a\in\{0,1\}^d$. In the situation of Example \ref{exa0} (d), the same estimator as in the situation of Example \ref{exa0} (b) is suggested, noticing that in this case the estimator can be rewritten as
\begin{equation*}
    \widehat{p}_n(a) =\frac{1}{n}\sum_{i=1}^n\1(I_i(1)=\cdots=I_i(d)=1)
\end{equation*}
for  $a\in\{0,1\}^d$ with $a(1)=\cdots=a(d)=1$,  
\begin{equation*}
    \widehat{p}_n(a) =\frac{1}{n}\sum_{i=1}^n\1(I_i(1)=\cdots=I_i(d)=0)
\end{equation*}
for  $a\in\{0,1\}^d$ with $a(1)=\cdots=a(d)=0$,
\begin{equation*}
    \widehat{p}_n(a) =\frac{1}{n}\sum_{i=1}^n\1(I_i(1)=\cdots=I_i(j)=1,I_i(j+1)=\cdots=I_i(d)=0)
\end{equation*}
for  $a\in\{0,1\}^d$ with $a(1)=\cdots=a(j)=1$ , $a(j+1)=\cdots=I_i(d)=0$, $j=1,\dots,d-1$, and zero else.

\end{example}

In fact, these estimators satisfy the following properties.

\begin{theorem}\label{thm1}
Suppose Assumption \ref{ass1a} and Assumption \ref{ass1b}  are valid. Then, the estimators in Example \ref{exa1} satisfy  Assumption \ref{ass2b} and Assumption \ref{ass3}. Considering  translations of the form $x\mapsto x+c$ for $c\in\R^d$ applied to $X_1,\dots,X_n$, the estimator $\widehat{\mu}_n$ in Example \ref{exa1} (a) is translation equivariant and the estimators $\widehat{\Sigma}_n$ and $\widehat{p}_n$  in Example \ref{exa1} (b) and  Example \ref{exa1} (c) are translation invariant.
\end{theorem}

Using  the estimators in Example \ref{exa1}, Theorem \ref{thm1} motivates to focus on a testing procedure that depends only on the estimators $\widehat{\Sigma}_n$ and  $\widehat{p}_n$ and on the empirically centered sample  
\begin{equation*}
(I_i,  I_i\odot(X_i-\widehat{\mu}_n)),~i=1,\dots,n,
\end{equation*}
noticing that these random variables can be obtained  from our original sample  $(I_i, I_i\odot X_i)$, $i=1,\dots,n$, due to
\begin{equation*}
 I_i\odot(X_i-\widehat{\mu}_n)= I_i\odot X_i - I_i\odot \widehat{\mu}_n,~i=1,\dots,n.
\end{equation*}
This results in  a testing procedure that  is invariant under translations of the form $x\mapsto x+c$ for $c\in\R^d$ applied to $X_1,\dots,X_n$ and in particular independent of the expectation vector $\mu$.

\section{Test statistic}\label{es}

 For more power and  flexibility in applications, our test incorporates a projection approach. For this purpose,  we intruduce with  $\Pi=\{\pi:\R^d\rightarrow\R^{d_\pi};\pi(x)= (x(i_1),\dots,x(i_{d_\pi}))^\top,x=(x(1),\dots,x(d))^\top\in\R^d, 1\leq i_1<\dots<i_{d_\pi}\leq d,1\leq d_\pi\leq d\}$, the set of projections of $\R^d$ onto $\R^{d_\pi}$, $1\leq d_\pi\leq d$. We adopt the application of  $\pi\in\Pi$ to occuring probability mass functions and matrices, e.g.,  $\pi(p)$ denotes the probability mass function of $\pi(I)$ and   $\pi(\Sigma)$ denotes the covariance matrix of $\pi(X)$.   Let us fix some  $\pi\in\Pi $ with $\pi:\R^d\rightarrow\R^{d_\pi}$ for a moment.  For $a\in\{0,1\}^{d_\pi}$,  we write  $D_a=\operatorname{diag}(a)\in\R^{d_\pi\times d_\pi}$ for the diagonal matrix with  diagonal entries $a$. Moreover, we write  $\Phi_{d_\pi}$ for the  multivariate standard normal distribution of dimension $d_\pi$.   Let 
\begin{align*}
\varphi_{\pi}(t)=\E\Big(\exp\big(i\scalar{t}{\pi(I\odot (X-{\mu}))}\big)\Big),~t\in \R^{d_\pi},
\end{align*}
be the characteristic function of the centered random vector $\pi(I\odot (X-{\mu}))$. Under the null hypothesis $H_0$, this characteristic function  is given by
\begin{align*}
    \phi_{\pi}(t)
    =\sum_{a\in\{0,1\}^{d_\pi}}\exp\left(-\frac{1}{2}\scalar{D_a \pi(\Sigma) D_a t}{t} \right) \pi(p)(a),~t\in\R^{d_\pi}.
\end{align*}
For that reason, it is $ \varphi_{\pi}=\phi_{\pi}$ under the null hypothesis $H_0$. Estimators are given by
\begin{equation*}
  \widehat  \varphi_{n,\pi}(t)=\frac{1}{n}\sum_{k=1}^n\exp\big(i\scalar{t}{\pi(I_k\odot (X_k-{\widehat\mu_n}))}\big),~t\in \R^{d_\pi},
\end{equation*}
that is the empirical characteristic function of the  empirically centered random vectors  $\pi(I_i\odot (X_i-\widehat\mu_n))$, $i=1,\dots,n$, and by
\begin{align*}
   \widehat \phi_{n,\pi}(t)
    =\sum_{a\in\{0,1\}^{d_\pi}}\exp\left(-\frac{1}{2}\scalar{D_a \pi(\widehat\Sigma_n) D_a t}{t} \right) \pi(\widehat p_n)(a),~t\in\R^{d_\pi}.
\end{align*}
Motivated by the usual BHEP test, we define a distance
\begin{equation*}
    T_{n,\pi}=n\int|  \widehat  \varphi_{n,\pi}-  \widehat \phi_{n,\pi}|^2\d\Phi_{d_\pi}.
\end{equation*}
A closed-form formula of $  T_{n,\pi}$ is available,  useful, e.g., for  implementation purposes. It is
\begin{align*}
      T_{n,\pi}=&\frac{1}{n}\sum_{i=1}^n\sum_{k =1}^n\exp\bigg(-\frac{1}{2}|\pi(I_i\odot (X_i-\widehat \mu_n))-\pi(I_k\odot (X_k-\widehat \mu_n))|^2\bigg)\\
    &-2\sum_{i=1}^n\sum_{a\in\{0,1\}^{d_\pi}}\E\bigg(\exp\Big(-\frac{1}{2}| W_{\pi,i,a}|^2\Big)\bigg)\pi(\widehat{p}_{n})(a)\\
    &+n\sum_{a\in\{0,1\}^{d_\pi}}\sum_{b\in\{0,1\}^{d_\pi}}\E\bigg(\exp\Big(-\frac 1 2 | Z_{\pi,a,b}|^2\Big)\bigg)\pi(\widehat{p}_{n})(a)\pi(\widehat{p}_{n})(b),
\end{align*}
where for $i=1,\dots,n$ and $a,b\in\{0,1\}^{d_\pi}$, the  random vectors in the formula satisfy $W_{\pi,i,a}\sim N_{d_\pi}(\pi(I_i\odot(X_i-\widehat{\mu}_n)),{D_a\pi(\widehat \Sigma_n) D_a})$ and $Z_{\pi,a,b}\sim N_{d_\pi}(0,{(D_a+D_{b})\pi(\widehat \Sigma_n) (D_a+D_{b})})$. Note that these random variables are defined on another probability space and related distributions are calculated as if the  sample $(I_i,I_i\odot X_i)$, $i=1,\dots,n$, is constant. Moreover, the expectations in the formula are the moment generating functions of  generalized  chi-squared distributions that can be obtained explicitly as it is explained in  \citet{DasGeisler-qChi2-param-conversion}. Introducing  
\begin{align*}
          V_{n,\pi}(t)=&\frac{1}{\sqrt{n}}\sum_{i=1}^n\big(\cos\scalar{t}{\pi(I_i\odot(X_i-\widehat{\mu}_n))} + \sin\scalar{t}{\pi(I_i\odot(X_i-\widehat{\mu}_n))}\big)\\
        &- \sqrt{n}\sum_{a\in\{0,1\}^{d_\pi}}\exp\bigg(-\frac{1}{2}\scalar{D_a\pi(\widehat{\Sigma}_n)D_at}{t}\bigg)\pi(\widehat{p}_n)(a),~t\in \R^{d_\pi},
\end{align*}
another alternative expression of  $T_{n,\pi}$ is obtained by 
\begin{equation*}
   T_{n,\pi}=\int V_{n,\pi}^2\d \Phi_{d_\pi}.
\end{equation*}
This expression can be used, e.g., for mathematical analysis or evaluation of $T_{n,\pi}$ by Monte-Carlo simulation. Finally, a BHEP test statistic is given by 
$$T_n=\sum_{\pi\in \Pi}w_{\pi}  T_{n,\pi},$$
where $w_{\pi}$, $\pi\in \Pi$, are weights that satisfy $w_{\pi}\in[0,1]$ for all $\pi\in\Pi$,  $\sum_{\pi\in\Pi}w_{\pi}=1$ and $w_{\rm{id}}>0$ with $\rm{id}\in\Pi$ as the identity map on  $\R^d$. Note that  singular covariance matrix estimates are allowed in the test statistic. The test  statistic  depends only on the estimators $\widehat{\Sigma}_n$ and  $\widehat{p}_n$ and on the empirically centered sample $(I_i,  I_i\odot(X_i-\widehat{\mu}_n))$, $i=1,\dots,n$. Using  the estimators in Example \ref{exa1}, the test statistic  is invariant under translations of the form $x\mapsto x+c$ for $c\in\R^d$ applied to $X_1,\dots,X_n$ and in particular independent of the expectation vector $\mu$.

\section{Asymptotic results}\label{as}

We present asymptotic results  as the sample size tends to infinity. At first, we establish an almost sure limit of the test statistic, available under the null hypothesis   $H_0$ and under the alternative $H_1$.

\begin{theorem}\label{thm2}
Suppose Assumption \ref{ass1a}, Assumption \ref{ass1b} and  Assumption \ref{ass2b} are valid. Then, we have
\begin{align*}
 \frac{1}{n}T_n\overset{a.s.}{\longrightarrow}  \kappa~\text{as}~n\to\infty,
\end{align*}
where
\begin{align*}
\kappa=\sum_{\pi\in \Pi}w_\pi\int|    \varphi_{\pi}-   \phi_{\pi}|^2\d\Phi_{d_\pi}.
\end{align*}
It is  $\kappa\in[0,\infty)$, where $\kappa= 0$ under the null hypothesis $H_0$ and  $\kappa>0$ under the alternative $H_1$.
\end{theorem}

Now, we state the limiting null distribution of the test statistic. Remembering the previous definition of $V_{n,\pi}$, we deal with a stochastic process  with values in the separable Hilbert space  $H_\pi=L^2(\R^{d_\pi},\BB^{d_\pi},\Phi_{d_\pi})$  consisting  of $(\BB^{d_\pi},\BB)$-measurable functions $f:\R^{d_\pi}\rightarrow \R$ that are square-integrable  with respect to $\Phi_{d_\pi}$ for all  $\pi\in\Pi$. For all  $\pi\in\Pi$,   we can alternatively deal with $ V_{n,\pi}$ as a stochastic process with values in the separable Hilbert space  $H_\Pi=L^2(\bigtimes_{\pi\in\Pi}\R^{d_\pi},\bigotimes_{\pi\in\Pi}\BB^{d_\pi},\bigotimes_{\pi\in\Pi}\Phi_{d_\pi})$ consisting  of $(\bigotimes_{\pi\in\Pi}\BB^{d_\pi},\BB)$-measurable functions $f:\bigtimes_{\pi\in\Pi}\R^{d_\pi}\rightarrow \R$ that are square-integrable  with respect to $\bigotimes_{\pi\in\Pi}\Phi_{d_\pi}$. Moreover, $(V_{n,\pi};\pi\in \Pi)$ can be regarded as a vector-valued stochastic process with values in the separable Hilbert space 
\begin{align*}
H=L^2\bigg(\bigtimes\limits_{\pi\in\Pi}\R^{d_\pi},\bigotimes\limits_{\pi\in\Pi}\BB^{d_\pi},\bigotimes\limits_{\pi\in\Pi}\Phi_{d_\pi}\bigg)^\Pi.
\end{align*}
This space is equipped with the scalar product $\scalar{f}{g}_{H}=\int \scalar{f}{g}\d\bigotimes_{\pi\in\Pi}\Phi_{d_\pi}$, $f,g\in H$, and with the norm $|f|_{H}=\sqrt{\scalar{f}{f}_{H}}=\sqrt{\int |f|^2 \d\bigotimes_{\pi\in\Pi}\Phi_{d_\pi}}$, $f\in H$. The test statistic can be rewritten as
\begin{equation*}
    T_n=\int \sum_{\pi\in \Pi}w_\pi V_{n,\pi}^2\d\bigotimes\limits_{\pi\in\Pi}\Phi_{d_\pi}.
\end{equation*}

\begin{theorem}\label{thm3}
Assume that the null hypothesis $H_0$ is valid. Suppose Assumption \ref{ass1a}, Assumption \ref{ass1b}, Assumption \ref{ass2b}  and  Assumption \ref{ass3} are valid. Then, we have
\begin{align*}
T_n\overset{d}{\longrightarrow} T~\text{as}~n\to\infty
\end{align*}
with a real-valued random variable
\begin{align*}
T=\int \sum_{\pi\in \Pi}w_\pi V_{\pi}^2\d\bigotimes\limits_{\pi\in\Pi}\Phi_{d_\pi},
\end{align*}
where $(V_{\pi};\pi\in \Pi)$ is a   vector-valued Gaussian process with values in $H$ and with expectation function given by
$$\E(V_{\pi}(t))=\E(\Psi(I,I\odot X,t,\pi))=0,~t\in\R^{d_\pi},~\pi\in \Pi,$$
and covariance function given by
$$\Cov(V_{\pi}(t),V_{\tau}(s))=\E(\Psi(I,I\odot X,t,\pi) \Psi(I,I\odot X,s,\tau)),~t\in\R^{d_\pi},~s\in\R^{d_\tau},~\pi,\tau\in \Pi,$$
\end{theorem}
with
\begin{align*}
\Psi(I,I\odot X,t,\pi)=& \cos\scalar{t}{\pi(I\odot(X-\mu))}\\
&+ \sin\scalar{t}{\pi(I\odot(X-\mu))}  \\
&-\sum_{a\in\{0,1\}^{d_\pi}}\exp\Big(-\frac{1}{2}\scalar{D_a\pi(\Sigma)D_a t}{t}\Big)\pi(p)(a)\\
&+\scalar{t}{\pi(\E(\sin\scalar{t}{\pi(I\odot(X-\mu)} I)\odot  l_\mu(I,I\odot X)}\\
 &-\sum_{a\in\{0,1\}^{d_\pi}}\frac{1}{2}\scalar{D_a\pi(l_\Sigma(I,I\odot X))D_a t}{t}\exp\Big(-\frac{1}{2}\scalar{D_a\pi(\Sigma)D_a t}{t}\Big)\pi(p)(a)\\
&+\sum_{a\in\{0,1\}^{d_\pi}}\pi(l_{p(a)}(I,I\odot X))
\exp\Big(-\frac{1}{2}\scalar{D_a\pi(\Sigma)D_a t}{t}\Big),~t\in\R^{d_\pi},~\pi\in \Pi.
\end{align*}

\section{Bootstrap procedure}\label{bp}

Given that $\alpha\in(0,1)$ is the significance level, denote by $c_{n,1-\alpha}$ a quantile of the distribution of the test statistic $\LL(T_n)$ under the null hypothesis $H_0$. Then, it is $P(T_n>c_{n,1-\alpha})\le \alpha$ under the null hypothesis  $H_0$. Supposing   Assumption \ref{ass1a}, Assumption \ref{ass1b}, Assumption \ref{ass2b} and  Assumption \ref{ass3} are valid, and  using  the same arguments as in (the proof of) Theorem 3 in \citet{GaigallWuebbolding_gBM-Test}, we can combine Theorem \ref{thm2} and Theorem \ref{thm3} to obtain  $P(T_n>c_{n,1-\alpha})\rightarrow 1$ as $n\to \infty$ under  the alternative $H_1$. These results  motivate to use $c_{n,1-\alpha}$ as critical value and to reject the null hypothesis $H_0$ if and only if $T_n>c_{n,1-\alpha}$. Unfortunately,  the distribution of the test statistic $\LL(T_n)$ under the null hypothesis $H_0$ depends on the  unknown underlying parameters $\mu$, $\Sigma$ and $p$ in general and the same applies to the related quantile $c_{n,1-\alpha}$ . For that reason, $c_{n,1-\alpha}$ is not available as critical value of a test in applications. To resolve this problem, a bootstrap procedure is suggested. For this purpse, let $X^*$ be a random vector of dimension $d$  with values in $\R^d$ and distribution $\mathcal N_d(\widehat \mu_n,\widehat \Sigma_n)$, let $I^*$ be a random vector of dimension $d$ with values in $\{0,1\}^d$ and probability mass function $\widehat p_n$ and let $X^*$ and $I^*$ be independent. Moreover, let 
$$(I_i^*,I_i^*\odot X_i^*),~i=1,\dots,n,$$
be a bootstrap sample of size $n\in\N$ of independent random variables, each with the same distribution as $(I^*,I^*\odot X^*)$. Note that these random variables are defined on another probability space and related distributions are calculated as if the  sample $(I_i,I_i\odot X_i)$, $i=1,\dots,n$, is constant.  Let $\widehat \mu_n^*$,  $\widehat \Sigma_n^*$ and  $\widehat p_n^*$ be bootstrap estimators, obtained now by the application of the respective functions to the bootstrap sample $(I_i^*,I_i^*\odot X_i^*)$, $i=1,\dots,n$. Letting
\begin{equation*}
  \widehat  \varphi_{n,\pi}^*(t)=\frac{1}{n}\sum_{k=1}^n\exp\big(i\scalar{t}{\pi(I_k^*\odot (X_k^*-{\widehat\mu_n^*}))}\big),~t\in \R^{d_\pi},
\end{equation*}
and
\begin{align*}
   \widehat \phi_{n,\pi}^*(t)
    =\sum_{a\in\{0,1\}^{d_\pi}}\exp\left(-\frac{1}{2}\scalar{D_a \pi(\widehat\Sigma_n^*) D_a t}{t} \right) \pi(\widehat p_n^*)(a),~t\in\R^{d_\pi},
\end{align*}
be the bootstrap versions of $\varphi_{n,\pi}$ and $   \phi_{n,\pi}$, we define by
\begin{equation*}
    T_{n,\pi}^*=n\int|  \widehat  \varphi_{n,\pi}^*-  \widehat \phi_{n,\pi}^*|^2\d\Phi_{d_\pi}
\end{equation*}
the related  distance and obtain the bootstrap test statistic as
\begin{equation*}
T_n^*=\sum_{\pi\in \Pi}w_{\pi}  T_{n,\pi}^*.
\end{equation*}
Finally, a bootstrap quantile of order $1-\alpha$ of the distribution of the bootstrap test statistic $\LL(T_n^*)$, denoted by   $c_{n,1-\alpha}^*$, is used as   critical value. In practice, the critical value can be obtained by  Monte-Carlo simulation. The bootstrap test rejects the null hypothesis $H_0$ if and only if $T_n>c_{n,1-\alpha}^*$.

\section{Simulations}\label{s}

To investigate the performance of the  testing procedure, simulation studies are performed.  In our simulation studies, the Monte-Carlo simulation is based on $2000$ replications. We fix  the significance level at $\alpha=0.05$ and we use $199$ bootstrap replications in each simulation to obtain the critical values of the tests. The null hypothesis under consideration is $ {\mathcal N}_d(\mu,\Sigma)$ with $\mu\in\R^d$ and $\Sigma\in\R^{d\times d}$ symmetric positive definite.  We choose $\mu=0$ and $\Sigma=\frac 1 2( {\bf I}_d+ {\bf 1}_d {\bf 1}_d^\top)$, where ${\bf 1}_d=(1,\dots,1)^\top$ and ${\bf I}_d=\operatorname{diag}({\bf 1}_d)$. For the implementation of our test, we treat the situation in Example \ref{exa0} (b) and use the corresponding estimators in  \ref{exa1}. We choose equal weights in the test statistic and treat the case that the components of $I$ are independent, each with the same binomial (Bernoulli) distribution with a given probability of missingness $\lambda\in[0,1)$. In our first simulation study, we consider different contamination  alternatives of the form $X=(1-B)Y+BZ$, where $B$ is a binomial (Bernoulli) random variable with probability $c\in[0,1]$, $Y$ is a random vector that follows the distribution from the null hypothesis,   $Z$ is a  $d$-dimensional random vector that follows an  alternative distribution and the random variables $B,Y,Z$ are independent. The values of  $c$ reflect the grade of contamination of the null hypothesis  by the alternative. As alternative distributions, we consider a multivariate Laplace distribution  $Z=\mu+\sqrt{V}\Sigma^\frac{1}{2}W$, where $V$ has a standard exponential distribution, $W$ has a $d$-dimensional standard normal distribution and $V$ and $W$ are independent, see \citet{Devroye_mLaplace}, as well as a multivariate $t$ distribution $Z=\mu+\sqrt{\frac{\nu}{V}}\Sigma^\frac{1}{2}W$, where $\nu\in\N$, $V$ has a $\chi^2$-distribution with $\nu$ degrees of freedom, $W$ has a $d$-dimensional standard normal distribution and $V$ and $W$ are independent, see \citet{Hofert_mt}. Furthermore, we treat the fixed probability of missingness $\lambda=0.106$. Empirical rejection rates (in \%) are displayed in Table \ref{tab:Simulation_Power}. Noticing that the settings with $c=0$ correspond to the case that the null hypothesis is true, we see that the test keeps the level of $\alpha=0.05$. Furthermore, the power of the test increases as the rate of contamination  increases and as  the sample size increases, which is  reasonable. For the  multivariate $t$ distribution as alternative, we see that the power decreases if the degrees of freedom increase. This is reasonable because  a multivariate $t$ distribution is more similar to  a multivariate normal distribution for larger  degrees of freedom. In cases with $c=1$, we have the same simulation settings as in \cite{Aleksić09122024}. Due to the problems mentioned that arise with this approach, there is only less comparability between both  tests.  Ignoring this limitations, we find that our test is comparatively  conservative and  cannot keep up with  the power values obtained in  \cite{Aleksić09122024} in this setting. This finding is in contrast to the settings in our next simulation study. Here, we examine the empirical power for higher probabilities of missingness $\lambda$. We consider two  classes of alternatives in the case $d=3$. The first one is given by a convolution of the form $X=Y+Z$, where   $Y$ is a random vector that follows the distribution from the null hypothesis and $Z$ follows the  uniform distribution on the centred cube with side length $2c$, $c\in[0,1]$, independent of $Y$. The second one  is an equal mixture of the form  $X=(1-B)Y+BZ$, where $B$ is a binomial (Bernoulli) random variable with probability $0.5$, $Y$ or $Z$ is a random vector that follows  a multivariate  normal distribution with expectation vector   $+  (2c,2c,2c)^\top$ or $-  (2c,2c,2c)^\top$, covariance matrix $\Sigma$ as stated  and the random variables $B,Y,Z$ are independent. Empirical rejection rates (in \%) are displayed in Table \ref{tab:new_alternatives}. Noticing again that the settings with $c=0$ correspond to the case that the null hypothesis is true, we see that the test keeps the level of $0.05$. The power of the test increases as the paramter $c$  increases and as  the sample size increases, which is  reasonable. We see that the power decreases if the probability of missingness increases. This is reasonable because  a increasing probability of missingness results in a decreasing amount of data effectively available. In particular, we see that our test is applicable and has power also for higher probabilities of missingness. In contrast, additional simulations conducted show that the test of  \cite{Aleksić09122024} is not able to detect the alternatives sufficiently in these settings.

    \begin{table}[htbp]
        \centering
        \begin{tabular}{rrrrrrrrrr}
\hline
            & & \multicolumn{2}{c}{Laplace alt.} & \multicolumn{6}{c}{$t$ alt.}\\
\hline
            & & $d=2$ & $d=3$ & \multicolumn{3}{c}{$d=2$} & \multicolumn{3}{c}{$d=3$} \\
\hline
             $n$ & $c$ &  &  & $\nu=5$ & $\nu=7$ & $\nu=11$ & $\nu=5$ & $\nu=7$ & $\nu=11$ \\
             \hline
             30&0.0& 5.75   &  5.20  & 5.50    & 4.05  & 4.55  & 5.10   & 5.05   & 4.90    \\
             & 0.2 & 7.45   &  5.65  & 8.80    & 7.65  & 6.50   & 10.05 & 5.30    & 4.85    \\
             & 0.4 & 11.30   &  9.80  & 12.50   & 9.60   & 5.80   & 13.55 & 8.60    & 5.80  \\ 
             & 0.6 & 17.40  & 18.40  & 17.25   & 10.75 & 8.70   & 16.15 & 11.45  & 6.35 \\
             & 0.8 & 30.30  & 27.55  & 20.95   & 13.35 & 8.25  & 21.30  & 11.95  & 7.85  \\
             & 1.0 & 36.45  & 39.65  & 22.10   & 14.10  & 9.95  & 23.50  & 13.75  & 8.75   \\
             \hline
             60 & 0.0 & 4.85   & 5.40  & 5.25   & 4.70   & 5.35  & 4.95   & 5.15  & 4.05    \\
                & 0.2 & 9.00   & 8.85  & 12.40   & 8.60   & 7.30   & 12.60   & 7.30   & 5.30  \\ 
                & 0.4 & 18.6  & 18.15 & 19.95  & 12.45 & 7.35  & 20.25  & 11.80  & 6.55   \\
                & 0.6 & 33.25  & 35.00 & 26.55  & 15.60  & 9.60   & 28.70   & 15.50  & 8.40    \\
                & 0.8 & 50.05  & 58.60 & 34.55  & 19.85 & 10.35 & 37.90   & 20.80  & 10.35  \\
                & 1.0 & 67.75  & 76.45 & 39.80   & 22.55 & 11.95 & 46.10   & 23.80  & 11.70    \\
             \hline
             90 & 0.0 & 6.20   & 4.80   & 5.05  & 5.65   & 5.55  & 5.70    & 5.25   & 5.40   \\
                & 0.2 & 10.15   & 10.70  & 13.90  & 11.20   & 7.45  & 14.75  & 8.65   & 7.00    \\ 
                & 0.4 & 24.60  & 26.45  & 26.40  & 14.35  & 9.50   & 26.75  & 14.15  & 7.65  \\ 
                & 0.6 & 46.00  & 50.60  & 35.95 & 19.50   & 9.90   & 37.80   & 21.00   & 10.80    \\ 
                & 0.8 & 70.05  & 77.60  & 47.60  & 25.35  & 12.80  & 51.55  & 28.40   & 13.75   \\
                & 1.0 & 84.10  & 92.10  & 53.35 & 31.70   & 14.15 & 63.35  & 36.60   & 14.50   \\
             \hline
            120 & 0.0 & 5.20  & 4.40   & 5.25   & 4.60   & 5.20   & 5.25   & 3.85   & 4.75   \\ 
                & 0.2 & 12.30   & 11.10  & 17.10   & 10.00    & 7.60   & 16.70   & 9.50    & 5.75   \\ 
                & 0.4 & 30.50  & 34.35  & 32.65  & 16.65 & 8.85  & 33.65  & 16.45  & 8.90  \\ 
                & 0.6 & 58.10  & 66.15  & 44.65  & 23.00    & 11.70  & 50.60   & 25.65  & 12.65   \\
                & 0.8 & 81.60  & 90.60  & 56.15  & 32.10  & 16.20  & 66.25  & 36.70   & 14.65   \\ 
                & 1.0 & 94.55  & 97.85  & 68.15  & 38.75 & 17.05 & 75.60   & 45.10   & 18.05 \\
\hline
       \end{tabular}
        \caption{Empirical rejection rates (in \%) of our test at level  $\alpha=0.05$ for  contamination alternatives with a multivariate  Laplace and a multivariate $t$ distribution with different degrees of freedom $\nu$, different dimensions $d$, different grades of contamination $c$ and different sample sizes $n$.}
        \label{tab:Simulation_Power}
    \end{table}

\begin{table}[htbp]
    \centering
    \begin{tabular}{rrrrrrrr}
\hline
       & & \multicolumn{3}{c}{Convolution alt.} & \multicolumn{3}{c}{Mixture alt.}\\
\hline
$n$&$c$&$\lambda=0.1$&$\lambda=0.2$&$\lambda=0.4$&$\lambda=0.1$&$\lambda=0.2$&$\lambda=0.4$\\
\hline
60&0.0&4.90&3.75&2.25&4.45&4.00&2.10\\
60&0.2&4.10&4.40&2.30&5.25&4.65&1.90\\
60&0.4&6.75&5.25&3.70&7.55&6.45&2.50\\
60&0.6&15.30&11.50&6.60&41.00&33.55&14.25\\
60&0.8&26.85&22.55&12.60&92.60&87.95&60.10\\
60&1.0&40.90&34.25&19.60&99.95&99.70&97.45\\
\hline
90&0.0&5.70&4.55&1.60&4.95&4.15&1.95\\
90&0.2&6.15&4.75&1.75&4.55&5.50&1.95\\
90&0.4&5.75&6.05&3.65&8.55&7.75&3.20\\
90&0.6&20.30&17.75&8.60&60.15&52.95&26.35\\
90&0.8&41.05&36.25&22.75&98.75&98.40&86.95\\
90&1.0&62.45&51.15&31.85&100.00&100.00&99.70\\
\hline
120&0.0&3.65&4.55&2.25&4.80&4.55&2.40\\
120&0.2&4.80&4.10&2.50&3.90&4.90&2.25\\
120&0.4&8.55&7.80&2.95&11.60&9.15&4.15\\
120&0.6&25.40&22.20&12.05&76.25&66.3&39.45\\
120&0.8&56.75&49.85&30.00&100.00&99.85&97.15\\
120&1.0&77.80&70.05&44.9&100.00&100.00&99.95\\
\hline
    \end{tabular}
    \caption{Empirical rejection rates (in \%) of our test at level  $\alpha=0.05$ for convolution and mixture alternatives with different parameters $c$, different sample sizes $n$ and different probabilities of missingness $\lambda$.}
    \label{tab:new_alternatives}
\end{table}

\color{black}

\section{Real data example}\label{rde}

As a real data example, we consider daily air quality measurements in New York for the time period  May 1, 1973 to September 30, 1973, given by the mean ozone (in parts per billion  and measured at Roosevelt Island from 13:00 to 15:00) and the solar radiation (in Langleys in the frequency band 4000-7700 Angstroms  and measured at Central Park from 8:00 to 12:00).  The dataset consists of 153 observations.  For the variables ozone and  solar radiation, the measurements for 37 and 7 days are missing, respectively, where in two cases,  both measurements are missing at the same time. The dataset is available in the statistical software R under the name airquality, see also \citet{chambers1983graphical}. For the implementation of our test, we treat the situation in Example \ref{exa0} (b) and use the corresponding estimators in  \ref{exa1}. Moreover, we choose equal weights in the test statistic. Based on 10000 bootstrap replications, we obtain an estimated $p$-value of 0.0227. The null  hypothesis of multivariate normality is rejected at the significance level $\alpha=0.05$.

\appendix

\section{Proofs}\label{p}

\begin{lemma}\label{lem1}
Given $k\in\N$, let $\vartheta\in\R^k$ and let $\widehat{\vartheta}_n$ be an estimator of  $\vartheta$ with values in $\R^k$, where the estimator is given by an appropriate measurable function applied to the sample   $(I_i,I_i\odot X_i)$, $i=1,\dots,n$. Assume
$$\widehat{\vartheta}_n\overset{a.s.}{\longrightarrow}  \vartheta~\text{as}~n\to\infty$$ 
and suppose the existence of a measurable map $l_\vartheta:\R^d\times\R^d\rightarrow \R^k$ with $\E(|l_\vartheta(I,I\odot X)|^2)<\infty$,  $\E(l_\vartheta(I,I\odot X))=0$  and 
\begin{equation*}
    \sqrt{n}(\widehat{\vartheta}_n-\vartheta)=\frac{1}{\sqrt{n}}\sum_{i=1}^n l_\vartheta(I_i,I_i\odot X_i)+\oP(1)~\text{as}~n\rightarrow\infty.
\end{equation*}
Given $\ell\in\N$, let $f:\R^k\rightarrow \R^\ell$ be differentiable in $\vartheta$. Then, we have
$$f(\widehat{\vartheta}_n)\overset{a.s.}{\longrightarrow}  f(\vartheta)~\text{as}~n\to\infty$$ 
and  the existence of a measurable map $l_{f(\vartheta)}:\R^d\times\R^d\rightarrow \R^\ell$ with $\E(|l_{f(\vartheta)}(I,I\odot X)|^2)<\infty$,  $\E(l_{f(\vartheta)}(I,I\odot X))=0$  and 
\begin{equation*}
    \sqrt{n}(f(\widehat{\vartheta}_n)-f(\vartheta))=\frac{1}{\sqrt{n}}\sum_{i=1}^n l_{f(\vartheta)}(I_i,I_i\odot X_i)+\oP(1)~\text{as}~n\rightarrow\infty.
\end{equation*}
\end{lemma}

\begin{proof}[Proof of Lemma \ref{lem1}]
Due to the continuity of $f$ in $\vartheta$, the first statement follows, that is $f(\widehat{\vartheta}_n)\overset{a.s.}{\rightarrow}  f(\vartheta)$ as $n\to\infty$. To show the second statement, denote by  $J_f(\vartheta)$ the derivative of $f$ in  $\vartheta$.   A Taylor explansion serves
\begin{equation*}
f(\widehat{\vartheta}_n)=f(\vartheta)+J_f(\vartheta)(\widehat{\vartheta}_n-\vartheta)+R(\widehat{\vartheta}_n)
\end{equation*}
with a remainder term $R(\widehat{\vartheta}_n)$ that satisfies 
\begin{equation*}
\frac{|R(\widehat{\vartheta}_n)|}{|\widehat{\vartheta}_n-\vartheta|}=\oP(1)~\text{as}~n\rightarrow\infty.
\end{equation*}
Combining this with Slutsky's theorem yields
\begin{align*}
   \sqrt{n}(f(\widehat{\vartheta}_n)-f(\vartheta))= & \sqrt{n}J_f(\vartheta)(\widehat{\vartheta}_n-\vartheta)+ \sqrt{n}R(\widehat{\vartheta}_n)\\
= &J_f(\vartheta)\sqrt{n}(\widehat{\vartheta}_n-\vartheta)+|\sqrt{n}(\widehat{\vartheta}_n-\vartheta)|\frac{R(\widehat{\vartheta}_n)}{| \widehat{\vartheta}_n-\vartheta|}\\
= &J_f(\vartheta)\frac{1}{\sqrt{n}}\sum_{i=1}^n l_\vartheta(I_i,I_i\odot X_i)+J_f(\vartheta)\oP(1)+|\sqrt{n}(\widehat{\vartheta}_n-\vartheta)|\oP(1)\\
=&\frac{1}{\sqrt{n}}\sum_{i=1}^n l_{f(\vartheta)}(I_i,I_i\odot X_i)+\oP(1)~\text{as}~n\rightarrow\infty,
\end{align*}
where $l_{f(\vartheta)}$ is defined by
\begin{align*}
   l_{f(\vartheta)}(I_i,I_i\odot X_i)=J_f(\vartheta) l_\vartheta(I_i,I_i\odot X_i)
\end{align*}
for $i=1,\dots,n$.  It is
$$\E(|l_{f(\vartheta)}(I,I\odot X)|^2)\le  |J_f(\vartheta)|^2\E(|l_\vartheta(I,I\odot X)|^2)<\infty$$
and
 $$\E(l_{f(\vartheta)}(I,I\odot X))=J_f(\vartheta)\E(l_\vartheta(I,I\odot X))=0.$$ 
This completes the proof.
\end{proof}

\begin{proof}[Proof of Theorem \ref{thm1}]
At first, let us consider the estimator $\widehat \mu_n$ in Example  \ref{exa1} (a).  We prove the stated translation equivariance  first. It is for all $c\in\R^d$ and all  $j=1,\dots,d$

\begin{align*}
  \widehat{\mu}_n(j)&=\frac{\sum_{i=1}^n I_i(j)(X_i(j)+c(j))}{\sum_{i=1}^nI_i(j)}\\
&=\frac{\sum_{i=1}^n I_i(j)X_i(j)}{\sum_{i=1}^nI_i(j)}+\frac{\sum_{i=1}^n I_i(j)c(j)}{\sum_{i=1}^nI_i(j)}\\
&=\frac{\sum_{i=1}^n I_i(j)X_i(j)}{\sum_{i=1}^nI_i(j)}+c(j)
\end{align*}
and so the equivariance  is valid. To show that Assumption \ref{ass2b}  holds, use the strong law of large numbers to see that
\begin{equation*}
    \widehat{\mu}_n(j)=\frac{\frac 1 n \sum_{i=1}^n I_i(j)X_i(j)}{\frac 1 n \sum_{i=1}^nI_i(j)}\overset{a.s.}{\longrightarrow}  \frac{P(I(j)=1)\E(X(j))}{P(I(j)=1)}=\mu(j)~\text{as}~n\to\infty
\end{equation*}
for all  $j=1,\dots,d$. Now, we show  the validity of Assumption \ref{ass3}. It is 
\begin{align*}
    \sqrt{n}(\widehat{\mu}_n(j)-\mu(j))=&\sqrt{n} \frac{\sum_{i=1}^nI_i(j)X_i(j)}{\sum_{i=1}^nI_i(j)} - \mu(j)\\
    =& \sqrt{n} \frac{\frac{1}{n}\sum_{i=1}^nI_i(j)(X_i(j)-\mu(j))}{\frac{1}{n}\sum_{i=1}^nI_i(j)} \\
    =& \sqrt{n} \frac{\frac{1}{n}\sum_{i=1}^nI_i(j)(X_i(j)-\mu(j))}{\frac{1}{n}\sum_{i=1}^nI_i(j)} - \sqrt{n} \frac{\frac{1}{n}\sum_{i=1}^nI_i(j)(X_i(j)-\mu(j))}{P(I(j)=1)}\\
     &+  \sqrt{n}\frac{\frac{1}{n}\sum_{i=1}^nI_i(j)(X_i(j)-\mu(j))}{P(I(j)=1)}\\
    =& \frac{1}{\sqrt{n}}\sum_{i=1}^n\frac{I_i(j)(X_i(j)-\mu(j))}{P(I(j)=1)}\\
    &-  \frac{ \sqrt{n}(  \frac{1}{n}\sum_{i=1}^nI_i(j)-P(I(j)=1)  )\frac{1}{n}\sum_{i=1}^nI_i(j)(X_i(j)-\mu(j))}{P(I(j)=1)\frac{1}{n}\sum_{i=1}^nI_i(j)}
\end{align*}
for all  $j=1,\dots,d$. From the central limit theorem, it is
$$\sqrt{n}\bigg(  \frac{1}{n}\sum_{i=1}^nI_i(j)-P(I(j)=1)  \bigg)\overset{d}{\longrightarrow}\mathcal N(0,P(I(j)=1)(1-P(I(j)=1)))~\text{as}~n\to\infty$$ 
for all  $j=1,\dots,d$, the   strong law of large numbers implies that
$$\frac{1}{n}\sum_{i=1}^nI_i(j)(X_i(j)-\mu(j))\overset{a.s.}{\longrightarrow}  0~\text{as}~n\to\infty$$
for all  $j=1,\dots,d$ and  Slutsky's theorem yields
$$\frac{ \sqrt{n}\left(  \frac{1}{n}\sum_{i=1}^nI_i(j)-P(I(j)=1)  \right)\frac{1}{n}\sum_{i=1}^nI_i(j)(X_i(j)-\mu(j))}{P(I(j)=1)\frac{1}{n}\sum_{i=1}^nI_i(j)}\overset{P}{\longrightarrow} 0~\text{as}~n\to\infty$$
for all  $j=1,\dots,d$.  Applying Slutsky's theorem again serves Assumption \ref{ass3} with $l_\mu$ given by
$$l_\mu(I_i,I_i\odot X_i)(j)=\frac{I_i(j) (X_i (j)- \mu(j))}{P(I(j)=1)}$$
for  $i=1,\dots,n$ and for all  $j=1,\dots,d$. Note that the moment conditions stated are obviously satisfied. Now, let us go to the estimator $\widehat \Sigma_n$ in Example  \ref{exa1} (b). The stated translation invariance is obviously true. At first, we show that Assumption \ref{ass2b} holds. Due to the invariance of this estimator under transformations of the form $x\mapsto x+c$ for $c\in\R^d$ applied to $X_1,\dots,X_n$, we suppose that $\mu=0$ without loss of generality.  Consider the matrix  $\widehat M_n$ defined in Example  \ref{exa1} (b). It follows  from the strong law of large numbers that
\begin{equation*}
    \widehat M_n\overset{a.s.}{\longrightarrow} M~\text{as}~n\to\infty,
\end{equation*}
where $M=\E(II^\top)$. Considering the matrix  $\widehat \Lambda_n$ defined in Example  \ref{exa1} (b), the strong law of large numbers and  the results fo the estimator $\widehat{\mu}_n$ in  Example  \ref{exa1} (a) imply
\begin{align*}
    \widehat{\Lambda}_n(j,k)=&\frac{1}{n}\sum_{i=1}^nI_i(j) (X_i(j)-\widehat{\mu}_n(j))I_i(k) (X_i(k)-\widehat{\mu}_n(k))\\
=&\frac{1}{n}\sum_{i=1}^nI_i(j) X_i(j)I_i(k) X_i(k)-\widehat{\mu}_n (j)\frac{1}{n}\sum_{i=1}^nI_i(j)I_i(k) X_i(k)\\
&-\widehat{\mu}_n(k)\frac{1}{n}\sum_{i=1}^nI_i(j)  X_i(j)I_i(k)+ \widehat{\mu}_n(j)\widehat{\mu}_n(k)\frac{1}{n}\sum_{i=1}^nI_i(j)I_i(k)\\
&\overset{a.s.}{\longrightarrow} \Lambda(j,k)~\text{as}~n\to\infty
\end{align*}
for all $j,k=1,\dots,d$, where $\Lambda=\E((I\odot X)(I\odot X)^\top)$. From the continuity of the related map, it follows that 
\begin{equation*}
    \widehat{S}_n=\widehat{\Lambda}_n\odot\widehat{M}_n^{\odot-1}\overset{a.s.}{\longrightarrow}S~\text{as}~n\to\infty,
\end{equation*}
where $S={\Lambda}\odot{M}^{\odot-1}$. Noticing that
\begin{align*}
S(j,k)&=\frac{\Lambda(j,k)}{ M(j,k)}=\frac{\E(I(j)X(j)I(k)X(k))}{\E(I(j)I(k))}=\E(X(j)X(k))=\Sigma(j,k)
\end{align*}
for all $j,k=1,\dots,d$, it follows from the continuity of the related map that
\begin{equation*}
    \widehat{\Sigma}_n=\sqrt{\widehat{S}_n\widehat{S}_n}\overset{a.s.}{\longrightarrow} \sqrt{SS}=S=\Sigma~\text{as}~n\to\infty
\end{equation*}
and so Assumption \ref{ass2b}  is valid.  Now, we show  the validity of Assumption \ref{ass3}. Obviously, the matrix $\widehat{M}_n$ satisfies
\begin{equation*}
    \sqrt{n}(\widehat{M}_n-M)=\frac{1}{\sqrt{n}}\sum_{i=1}^nl_M(I_i,I_i\odot X_i),
\end{equation*}
where $  l_M$ is given by
\begin{equation*}
    l_M(I_i,I_i\odot X_i)= I_iI_i^\top - M
\end{equation*}
for $i=1,\dots,n$. Moreover, the matrix $\widehat{\Lambda}_n$ satisfies
\begin{align*}
  \sqrt{n}(\widehat{\Lambda}_n(j,k)-\Lambda(j,k))=&\sqrt{n} \bigg( \frac{1}{n}\sum_{i=1}^n \big(I_i(j)(X_i(j)-\widehat{\mu}_n(j))\big)\big(I_i(k)(X_i(k)-\widehat{\mu}_n(k))\big) - \Lambda(j,k) \bigg)\\
    =& \frac{1}{\sqrt{n}}\sum_{i=1}^n (I_i(j)X_i(j)I_i(k)X_i(k) - \Lambda(j,k))\\
& -\sqrt{n}\widehat{\mu}_n (j)\frac{1}{n}\sum_{i=1}^nI_i(j)I_i(k) X_i(k)-\sqrt{n}\widehat{\mu}_n(k)\frac{1}{n}\sum_{i=1}^nI_i(j)  X_i(j)I_i(k)\\
&+ \sqrt{n}\widehat{\mu}_n(j)\widehat{\mu}_n(k)\frac{1}{n}\sum_{i=1}^nI_i(j)I_i(k)
\end{align*}
for all $j,k=1,\dots,d$.  Combining the strong law of large numbers,  the results for the estimator $\widehat{\mu}_n$ in  Example  \ref{exa1} (a) and  Slutsky's theorem serves 
\begin{align*}
  \sqrt{n}(\widehat{\Lambda}_n-\Lambda)=\frac{1}{\sqrt{n}}\sum_{i=1}^n l_{\Lambda}(I_i,I_i\odot X_i)+\oP(1)~\text{as}~n\rightarrow\infty,
\end{align*}
where $ l_\Lambda$ is given by
\begin{equation*}
    l_\Lambda(I_i,I_i\odot X_i)=(I_i\odot X_i)(I_i\odot X_i)^\top-\Lambda
\end{equation*}
for $i=1,\dots,n$. Consider the  maps $f$ and $g$, given by  $f(S)=\sqrt {SS}$ and $g(\Lambda,M)={\Lambda}\odot{M}^{\odot-1}$, that are  differentiable in the underlying paramters $S=\Sigma$ and  $(\Lambda,M)$ with $S=\Sigma={\Lambda}\odot{M}^{\odot-1}$. The composition $h=f \circ g$ is also   differentiable in  $(\Lambda,M)$.  Dealing with the occuring matrices or pairs of matrices as vectors for a moment, that can be obtained, e.g., by applicaiton of the vectorization operation, and applying Lemma \ref{lem1}, shows that Assumption \ref{ass3} is satisfied. For the estimators $\widehat{p}_n$ in  Example  \ref{exa1} (a), the stated translation invariance, Assumption \ref{ass2b} and Assumption \ref{ass3} are  valid. These statements are either obvious or can be easily obtained  from Lemma \ref{lem1}.
\end{proof}

\begin{proof}[Proof of Theorem \ref{thm2}]
Let us suppose that the null hypothesis $H_0$ is valid. Let us fix some arbitrary $\pi\in\Pi$. For  $k=1,\dots,n$ and all $t\in\R^{d_\pi}$, a first order Taylor expansion of the functions $\sin$ and $\cos$ about $\scalar{\pi(I_k\odot(X_k-\mu))}{t}$ yields (random) numbers $\xi_{n,\pi,k}(t)$ and $\eta_{n,\pi,k}(t)$ between  $\scalar{\pi(I_k\odot(X_k-\widehat{\mu}_n))}{t}$ and $\scalar{\pi(I_k\odot(X_k-\mu))}{t}$ such that
    \begin{align*}
     &\exp(i\scalar{\pi(I_k\odot(X_k-\widehat{\mu}_n))}{t})\\
        =&\cos\scalar{\pi(I_k\odot(X_k-\widehat{\mu}_n))}{t}+i\sin\scalar{\pi(I_k\odot(X_k-\widehat{\mu}_n))}{t}\\
        =&\cos\scalar{\pi(I_k\odot(X_k-\mu))}{t}-\sin(\xi_{n,\pi,k}(t))\scalar{\pi(I_k\odot(\mu-\widehat{\mu}_n)}{t}\\
        &+i\sin\scalar{\pi(I_k\odot(X_k-\mu))}{t}+i\cos(\eta_{n,\pi,k}(t))\scalar{\pi(I_k\odot(\mu-\widehat{\mu}_n))}{t}\\
        =&\exp(i\scalar{\pi(I_k\odot(X_k-\mu))}{t})+\scalar{\pi(I_k\odot(\mu-\widehat{\mu}_n))}{t}(i\cos(\eta_{n,\pi,k}(t))-\sin(\xi_{n,\pi,k}(t))).
    \end{align*}
Setting
    \begin{align*}
 \varphi_{n,\pi}(t)=\frac{1}{n}\sum_{k=1}^n\exp(i\scalar{\pi(I_k\odot(X_k-\mu))}{t}), ~t\in\R^{d_\pi},
    \end{align*}
we obtain for  all $t\in\R^{d_\pi}$ that
    \begin{align}
&|  \widehat  \varphi_{n,\pi}(t)-  \widehat \phi_{n,\pi}(t)-   (\varphi_{\pi}(t)-   \phi_{\pi}(t))|\notag\\
        =&| \varphi_{n,\pi}(t)-   \varphi_{\pi}(t)- ( \widehat \phi_{n,\pi}(t)-  \phi_{\pi}(t))\notag\\
&+\frac{1}{n}\sum_{k=1}^n\scalar{\pi(I_k\odot(\mu-\widehat{\mu}_n))}{t}(i\cos(\eta_{n,\pi,k}(t))-\sin(\xi_{n,\pi,k}(t)))|\notag \\
        \leq&| \varphi_{n,\pi}(t)-\varphi_{\pi}(t)|\label{eq:AS_Limit_ugl_1}\\
        &+|\widehat \phi_{n,\pi}(t)-    \phi_{\pi}(t)|\label{eq:AS_Limit_ugl_2}\\
        &+\bigg|\frac{1}{n}\sum_{k=1}^n\scalar{\pi(I_k\odot(\mu-\widehat{\mu}_n))}{t}(i\cos(\eta_{n,\pi,k}(t))-\sin(\xi_{n,\pi,k}(t)))\bigg|.\label{eq:AS_Limit_ugl_3}
    \end{align}
Let us treat  the three terms in \eqref{eq:AS_Limit_ugl_1},  \eqref{eq:AS_Limit_ugl_2} and  \eqref{eq:AS_Limit_ugl_3} separately. Let us start with the treatment of the term \eqref{eq:AS_Limit_ugl_2} first. There exist measurable sets $\Omega_\mu$, $\Omega_\Sigma$ and $\Omega_p$ with $P(\Omega_\mu)=1$, $P(\Omega_\Sigma)=1$ and $P(\Omega_p)=1$ such that   $\widehat \mu_n\rightarrow \mu$ as $n\to \infty$ on $\Omega_\mu$,  $\widehat \Sigma_n\rightarrow\Sigma$ as $n\to \infty$  on $\Omega_\Sigma$ and  $\widehat p_n\rightarrow p$   as $n\to \infty$ on $\Omega_p$. From the continuity of the related map, we have  for  all $t\in\R^{d_\pi}$
    \begin{align*}
        |\widehat \phi_{n,\pi}(t)-    \phi_{\pi}(t)|\longrightarrow 0~\text{as}~n\to \infty
    \end{align*}
 on $\Omega_\Sigma\cap \Omega_p$. Now, let us treat  the term \eqref{eq:AS_Limit_ugl_3}. From Cauchy-Schwarz inequality, it is  for  all $t\in\R^{d_\pi}$ 
  \begin{align*}
        &\bigg|\frac{1}{n}\sum_{k=1}^n\scalar{\pi(I_k\odot(\mu-\widehat{\mu}_n))}{t}(i\cos(\eta_{n,\pi,k}(t))-\sin(\xi_{n,\pi,k}(t)))\bigg|\\
        \leq&2\frac{1}{n}\sum_{k=1}^n\bigg|\scalar{\pi(I_k\odot(\mu-\widehat{\mu}_n))}{t}\bigg|\\
        \leq&2|\pi(\mu-\widehat{\mu}_n)||t| \longrightarrow 0~\text{as}~n\to \infty
    \end{align*}
 on $\Omega_\mu$. It remains to treat  the term \eqref{eq:AS_Limit_ugl_1}. For this purpose, let $D$ be a countable dense subset of $\R^{d_\pi}$, e.g., $D=\Q^{d_\pi}$. From the strong law of large numbers, we have for each $s\in D$ a measurable set $\Omega_s$ with $P(\Omega_s)=1$ and
    \begin{equation*}
       | \varphi_{n,\pi}(s)-\varphi_{\pi}(s)| \longrightarrow 0~\text{as}~n\to \infty
    \end{equation*}
 on $\Omega_s$. Now let  $t\in\R^{d_\pi}$  and $s\in D$ be arbitrary. Then, it is
     \begin{equation}\label{eq:AS_Limit_ugl_4}
        | \varphi_{n,\pi}(t)-\varphi_{\pi}(t)| \leq |\varphi_{n,\pi}(t)-\varphi_{n,\pi}(s)|+|\varphi_{n,\pi}(s)-\varphi_{\pi}(s)|+|\varphi_{\pi}(s)-\varphi_{\pi}(t)|.
    \end{equation}
For the middle term on the right hand side of the  inequality \eqref{eq:AS_Limit_ugl_4}, we have
  \begin{equation*}
       | \varphi_{n,\pi}(s)-\varphi_{\pi}(s)| \longrightarrow 0~\text{as}~n\to \infty
    \end{equation*}
 on $\cap_{s\in D}\Omega_s$. For the last term on the right hand side of the  inequality \eqref{eq:AS_Limit_ugl_4}, we have
   \begin{align*}
      |\varphi_{\pi}(s)-\varphi_{\pi}(t)|=& |\E(\exp(i\scalar{\pi(I\odot(X-\mu))}{s}-\exp(i\scalar{\pi(I\odot(X-\mu))}{t})))|\\
        \leq&|s-t|\E(|\pi(I\odot(X-\mu))|),
    \end{align*}
where   Cauchy-Schwarz inequality and $|\exp(ix)-\exp(iy)|\leq|x-y|$, $x,y\in\R$, is used. Analogously, for the first term on the right hand side of the  inequality \eqref{eq:AS_Limit_ugl_4}, we have
       \begin{align*}
      |\varphi_{n,\pi}(t)-\varphi_{n,\pi}(s)|\leq& \frac{1}{n}\sum_{k=1}^n|\scalar{\pi(I_k\odot(X_k-\mu))}{t-s}|\\
        \leq&|s-t|\frac{1}{n}\sum_{k=1}^n|\pi(I_k\odot(X_k-\mu))| \\
&\longrightarrow |s-t|\E(|\pi(I\odot(X-\mu))|)~\text{as}~n\to \infty
    \end{align*}
on a measurable set $\Omega$  with $P(\Omega)=1$ due to the strong law of large numbers.  Because $s\in D$ is arbitrary and $D$ is a dense subset of $ \R^{d_\pi}$, it follows with $s\to t$ that
  \begin{equation*}
        | \varphi_{n,\pi}(t)-\varphi_{\pi}(t)| \longrightarrow 0~\text{as}~n\to \infty
    \end{equation*}
 on $\cap_{s\in D}\Omega_s\cap \Omega$. Setting $\Omega'=\cap_{s\in D}\Omega_s\cap\Omega_\mu\cap\Omega_\Sigma\cap \Omega_p\cap \Omega$, it is $\Omega'$ a measurable set with $P(\Omega')=1$ such that for all   $t\in\R^{d_\pi}$ 
    \begin{align*}
&\widehat  \varphi_{n,\pi}(t)-  \widehat \phi_{n,\pi}(t)\longrightarrow \varphi_{\pi}(t)-   \phi_{\pi}(t)~\text{as}~n\to \infty
    \end{align*}
on $\Omega'$. Because $\pi\in\Pi$ is arbitrary and $\sup_{t\in\R^{d_\pi}}|\widehat  \varphi_{n,\pi}(t)-  \widehat \phi_{n,\pi}(t)|^2\le 4$, it follows from the dominated convergence theorem that 
\begin{align*}
 \frac{1}{n}T_n=\sum_{\pi\in \Pi}w_\pi\int|   \widehat \varphi_{n,\pi}-  \widehat \phi_{n,\pi}|^2\d\Phi_{d_\pi}\overset{a.s.}{\longrightarrow} \sum_{\pi\in \Pi}w_\pi\int|    \varphi_{\pi}-   \phi_{\pi}|^2\d\Phi_{d_\pi}= \kappa~\text{as}~n\to\infty.
\end{align*}
Without loss of generality, let $\mu=0$. For every $\pi\in\Pi$, it is $ \varphi_{\pi}$  the characteristic function of  $\pi(I\odot X)$ and $ \phi_{\pi}$ the characteristic function of  $\pi(I\odot X)$ under the null hypothesis $H_0$. Under the null hypothesis $H_0$, it is $ \varphi_{\pi}=  \phi_{\pi}$ for all  $\pi\in\Pi$ and so $\kappa=0$. Now, assume that the alternative $H_1$ is valid. Due to $w_{\rm{id}}>0$, it is sufficient to show that $ \varphi_{\rm{id}}(t)\neq\phi_{\rm{id}}(t)$ for at least one $t\in\R^d$. Noticing that the charateristic function is in general continuous, this   implies  $\kappa\ge w_{\rm{id}}\int|    \varphi_{\rm{id}}-   \phi_{\rm{id}}|^2\d\Phi_{d}>0$,   see (the proof of) Theorem 3 in \citet{GaigallWuebbolding_gBM-Test} for details. As it is already seen at the beginning of this work, the distribution $\LL(I,I\odot X)$ uniquely determines the distribution  $\LL(X)$. This completes the proof.
\end{proof}

\begin{proof}[Proof of Theorem \ref{thm2}]

Let us fix some arbitrary $\pi\in\Pi$.  For the proof, remember the  previously introduced expression   $T_{n,\pi}=\int V_{n,\pi}^2\d \Phi_{d_\pi}$. For  $i=1,\dots,n$ and all $t\in\R^{d_\pi}$, second order Taylor expansions  about $\scalar{\pi(I_i\odot(X_i-\mu))}{t}$ yield (random) numbers $\xi_{n,\pi,i}(t)$ and $\eta_{n,\pi,i}(t)$ between  $\scalar{\pi(I_i\odot(X_i-\widehat{\mu}_n))}{t}$ and $\scalar{\pi(I_i\odot(X_i-\mu))}{t}$ such that
    \begin{align*}
           & \cos\scalar{t}{\pi(I_i\odot(X_i-\widehat{\mu}_n))}+ \sin\scalar{t}{\pi(I_i\odot(X_i-\widehat{\mu}_n))}\\
=&\cos\scalar{t}{\pi(I_i\odot(X_i-\mu))}-\sin\scalar{t}{\pi(I_i\odot(X_i-\mu)}\scalar{t}{\pi(I_i\odot(\mu-\widehat{\mu}_n)}\\
&-\frac{1}{2}\scalar{t}{\pi(I_i\odot(\mu-\widehat{\mu}_n))}^2\cos(\eta_{n,\pi,i}(t)) +\sin\scalar{t}{\pi(I_i\odot(X_i-\mu))}\\
&+\cos\scalar{t}{\pi(I_i\odot(X_i-\mu)}\scalar{t}{\pi(I_i\odot(\mu-\widehat{\mu}_n)}-\frac{1}{2}\scalar{t}{\pi(I_i\odot(\mu-\widehat{\mu}_n))}^2\sin(\xi_{n,\pi,i}(t)).
      \end{align*}
Similarily, for  all $t\in\R^{d_\pi}$ and all $a\in\{0,1\}^{d_\pi}$, a second order Taylor expansion about  $(\scalar{D_a\pi(\Sigma)D_a t}{t},\pi(p)(a))^\top$  yields a (random) number  $\zeta_{n,\pi,a}(t)$ in $[0,1]$ such that
    \begin{align*}
            &\exp\Big(-\frac{1}{2}\scalar{D_a\pi(\widehat{\Sigma}_n)D_a t}{t}\Big)\pi(\widehat{p}_n)(a)\\
            =&\exp\Big(-\frac{1}{2}\scalar{D_a\pi(\Sigma)D_a t}{t}\Big)\pi(p)(a)\\
            &-\frac{1}{2}\scalar{D_a\pi(\Sigma-\widehat{\Sigma}_n)D_a t}{t}\exp\Big(-\frac{1}{2}\scalar{D_a\pi(\Sigma)D_a t}{t}\Big)\pi(p)(a)\\
            &+(\pi(p-\widehat{p}_n)(a))\exp\Big(-\frac{1}{2}\scalar{D_a\pi(\Sigma)D_a t}{t}\Big) \\
&+\frac{1}{8}\scalar{D_a\pi(\Sigma-\widehat{\Sigma}_n)D_a t}{t}^2\\
&\exp\Big(-\frac{1}{2}\scalar{(D_a\pi(\Sigma)D_a-\zeta_{n,\pi,a}(t)D_a\pi(\Sigma-\widehat{\Sigma}_n)D_a)t}{t}\Big)\\
            &\pi(p-\zeta_{n,\pi,a}(t)(p-\widehat{p}_n))(a)\\
            &-\frac{1}{2}\scalar{D_a\pi(\Sigma-\widehat{\Sigma}_n)D_a t}{t}\pi(p-\widehat{p}_n)(a)\\
&\exp\Big(-\frac{1}{2}\scalar{(D_a\pi(\Sigma)D_a-\zeta_{n,\pi,a}(t)D_a\pi(\Sigma-\widehat{\Sigma}_n)D_a)t}{t}\Big).
    \end{align*}
Combining the above we obtain for all $t\in\R^{d_\pi}$
    \begin{align*}
          V_{n,\pi}(t)=&\frac{1}{\sqrt{n}}\sum_{i=1}^n\Psi(I_i,I_i\odot X_i,t,\pi)+R_{n,\pi}(t)
    \end{align*}
with a remainder term $R_{n,\pi}(t)$. Noticing that for  $i=1,\dots,n$ and all $t\in\R^{d_\pi}$
    \begin{align*}
&-\frac{1}{\sqrt{n}}\sum_{i=1}^n\sin\scalar{t}{\pi(I_i\odot(X_i-\mu)}\scalar{t}{\pi(I_i\odot(\mu-\widehat{\mu}_n)}\\
&=\frac{1}{\sqrt{n}}\sum_{i=1}^n\scalar{t}{\pi(\E(\sin\scalar{t}{\pi(I\odot(X-\mu)} I)\odot  l_\mu(I,I\odot X)}+\oP(1)~\text{as}~n\rightarrow\infty
    \end{align*}
and
    \begin{align*}
&\frac{1}{\sqrt{n}}\sum_{i=1}^n\cos\scalar{t}{\pi(I_i\odot(X_i-\mu)}\scalar{t}{\pi(I_i\odot(\mu-\widehat{\mu}_n)}=\oP(1)~\text{as}~n\rightarrow\infty
    \end{align*}
under the null hypothesis $H_0$ and  that the matrix
    \begin{align*}
D_a\pi(\Sigma)D_a-\zeta_{n,\pi,a}(t)D_a\pi(\Sigma-\widehat{\Sigma}_n)D_a=D_a\pi(\Sigma)D_a(1-\zeta_{n,\pi,a}(t))+\zeta_{n,\pi,a}(t)D_a\widehat{\Sigma}_nD_a
    \end{align*}
is a convex combination of positive semidefinite matricies als so again a positive semidefinite matrix, it is easy to see that  the remainder term satisfies $|R_{n,\pi}|_{H_\pi}\overset{P}{\rightarrow} 0$ and likewise $|R_{n,\pi}|_{H_\Pi}\overset{P}{\rightarrow} 0$ as $n\to \infty$. Because the null hypothesis $H_0$ is valid, we have
    \begin{align*}
\E( \cos\scalar{t}{\pi(I\odot(X-\mu))})=0
    \end{align*}
and 
    \begin{align*}
\E( \sin\scalar{t}{\pi(I\odot(X-\mu))} =\sum_{a\in\{0,1\}^{d_\pi}}\exp\Big(-\frac{1}{2}\scalar{D_a\pi(\Sigma)D_a t}{t}\Big)\pi(p)(a)
    \end{align*}
for all $t\in\R^{d_\pi}$ and it follows that  $\E(\Psi(I,I\odot X,t,\pi))=0$. It is for all   $t\in\R^{d_\pi}$
\begin{align*}
        (V_{\pi}(t);\pi\in \Pi)=\frac{1}{\sqrt{n}}\sum_{i=1}^n(\Psi(I_i,I_i\odot X_i,t,\pi);\pi\in \Pi)+R_{n}(t)
    \end{align*}
a vector-valued stochastic process with values in the separable Hilbert space $H$ and $R_{n}(t)$ a remainder term  that satisfies $|R_{n,\pi}|_{H}\overset{P}{\rightarrow} 0$ as $n\to \infty$. A general version of the central limit theorem in separable Hilbert spaces, see Theorem 7.5.1 in  \citet{LahaRohatgi}, implies
\begin{align*}
        (V_{\pi};\pi\in \Pi)\overset{d}{\longrightarrow} (V_{\pi};\pi\in \Pi)~\text{as}~n\to \infty,
    \end{align*}
where $(V_{\pi};\pi\in \Pi)$ is a   vector-valued Gaussian process with values in $H$ and with expectation function given by
$$\E(V_{\pi}(t))=\E(\Psi(I,I\odot X,t,\pi))=0,~t\in\R^{d_\pi},~\pi\in \Pi,$$
and covariance function given by
$$\Cov(V_{\pi}(t),V_{\tau}(s))=\E(\Psi(I,I\odot X,t,\pi) \Psi(I,I\odot X,s,\tau)),~t\in\R^{d_\pi},~s\in\R^{d_\tau},~\pi,\tau\in \Pi.$$
Because the map $(v_{\pi};\pi\in \Pi)\mapsto \int \sum_{\pi\in \Pi}w_\pi v_{\pi}^2\d\bigotimes\limits_{\pi\in\Pi}\Phi_{d_\pi}$, $(v_{\pi};\pi\in \Pi)\in H$, is continuous, the stated convergence in distribution of the test statistic follows from the continuous mapping theorem.
\end{proof}

\bibliographystyle{apacite}


\end{document}